\documentclass[11pt,leqno]{amsart}

\usepackage{hyperref}
\usepackage{mathrsfs}
\usepackage{tikz-cd}
\usepackage{amsmath,amsfonts,amssymb}
\usepackage{xfrac}
\usepackage{nicefrac}
\usepackage{color}
\usepackage[draft]{fixme}

\usepackage{graphicx}

\newtheorem{thm}{Theorem}
\newtheorem{theorem}{Theorem}[section]
\newtheorem{lem}[theorem]{Lemma}
\newtheorem{pro}[theorem]{Proposition}
\newtheorem{cor}[theorem]{Corollary}

\theoremstyle{definition}
\newtheorem{definition}[theorem]{Definition}

\numberwithin{equation}{section}

\DeclareMathOperator{\lev}{\mathcal{L}}
\DeclareMathOperator{\st}{st}
\DeclareMathOperator{\rst}{rst}
\DeclareMathOperator{\sym}{Sym}
\DeclareMathOperator{\Z}{\mathbb{Z}}
\DeclareMathOperator{\N}{\mathbb{N}}
\newcommand{\E}{\mathcal{E}}
\newcommand{\F}{\mathcal{F}}
\newcommand{\G}{\mathcal{G}}
\newcommand{\T}{\mathcal{T}}

\DeclareMathOperator{\Ima}{im}
\DeclareMathOperator{\Aut}{Aut}

\begin{document}

\title{GGS-groups over primary trees: Branch structures}
\author{E. Di Domenico}
\address{Department of Mathematics, University of Trento, Trento, Italy/Department of Mathematics, University of the Basque Country, Bilbao, Spain}
%\curraddr{Department of Mathematics and Statistics,
%Case Western Reserve University, Cleveland, Ohio 43403}
\email{elena.didomenico@unitn.it}
\thanks{The first author was partly supported by the National Group for Algebraic and Geometric Structures, and their Applications (GNSAGA--INdAM), and also acknowledges support from the Department of Mathematics of the University of Trento.
The first and second authors are supported by the Spanish Government, grant PID2020-117281GB-I00, partly with FEDER funds, and by the Basque Government, grant IT483-22.
The first and third authors are members of GNSAGA--INdAM}

\author{G. A. Fern\'{a}ndez-Alcober}
\address{Department of Mathematics, University of the Basque Country, Bilbao, Spain}
\email{gustavo.fernandez@ehu.eus}


\author{N. Gavioli}
\address{Department of Information Engineering, Computer Science and Mathematics, L'Aquila, Italy}
%\curraddr{Department of Mathematics and Statistics,
%Case Western Reserve University, Cleveland, Ohio 43403}
\email{norberto.gavioli@univaq.it}
%\thanks{The second author was supported by}
%    General info
\subjclass[2000]{Primary 20E08; Secondary 20E26}

%\date{January 1, 2001 and, in revised form, June 22, 2001.}


\keywords{Group theory, automorphisms of rooted trees, branch groups, weakly branch groups}

\begin{abstract}
We study branch structures in Grigorchuk-Gupta-Sidki groups (GGS-groups) over primary trees, that is, regular rooted trees of degree $p^n$ for a prime $p$.
Apart from a small set of exceptions for $p=2$, we prove that all these groups are weakly regular branch over $G''$.
Furthermore, in most cases they are actually regular branch over $\gamma_3(G)$.
This is a significant extension of previously known results regarding periodic GGS-groups over primary trees and general GGS-groups in the case $n=1$.
We also show that, as in the case $n=1$, a GGS-group generated by a constant vector is not branch.
\end{abstract}

\maketitle

\section{Introduction}

Groups of automorphisms of regular rooted trees are a rich source of examples with very interesting properties in group theory.
The first Grigorchuk group, defined by Grigorchuk in 1980 \cite{Grigorcuk1980},
is one of the first instances of an infinite finitely generated periodic group, thus providing a negative solution to the General Burnside  Problem.
It is also the first example of a group with intermediate growth \cite{Grigorchuk1983}, hence solving the Milnor Problem \cite{Carlitz1968}. 
Many other groups of automorphisms of rooted trees have since been defined and studied.
Prominent  examples are the Gupta-Sidki $p$-groups~\cite{Gupta1983}, for $p$ an odd prime, and the second Grigorchuk group~\cite{Grigorcuk1980}, which belong to the family of the so-called
\emph{Grigorchuk-Gupta-Sidki groups} (GGS-groups, for short).
This paper is devoted to the study of branch structures in GGS-groups over primary trees, extending results of Vovkivsky \cite{Vovkivsky2000}, Fern\'andez-Alcober and Zugadi-Reizabal \cite{Fernandez-Alcober2014}, and Fern\'andez-Alcober, Garrido, and Uria-Albizuri
\cite{Fernandez-Alcober2017}.
Before stating our main results we quickly introduce all relevant concepts.

Let $m\ge 2$ be an integer and let $X=\{x_1,\ldots,x_m\}$ be a set with $m$ elements.
We write $X^*$ for the free monoid generated by $X$.
The \emph{descendants} of a word $u\in X^*$ are the words $v=uz$ with
$z\in X^*$, and $v$ is an \emph{immediate descendant} of $u$ if $z\in X$.
The \emph{$m$-adic tree} $\T$ is the rooted tree whose vertices are the words in $X^*$ (the root being the empty word $\varnothing$), and where two vertices are connected by an edge if any of the two is an immediate descendant of the other.
The $\ell$-th level $\lev_{\ell}$ of $\T$ consists of all words of length $\ell$ in $X^*$.
If $m$ is a power of a prime $p$, we say that $\T$ is a \emph{primary tree}.

The automorphisms of $\T$ as a graph form a group $\Aut\T$ under composition.
Given $f\in \Aut\T$ and a vertex $u$ of $\T$, we can write $f(uz)=f(u)f_u(z)$ for all
$z\in X^*$,  where $f_u\in  \Aut\T$ is called the \emph{section of $f$ at $u$}.
For every $\ell\ge 1$, we let $\st(\ell)$ denote the pointwise stabilizer of $\lev_{\ell}$ in
$\Aut\T$.
Then we have an isomorphism
\[
\begin{matrix}
\psi_{\ell} & \colon & \st(\ell) &
\longrightarrow & \Aut\T \times \overset{m^{\ell}}{\cdots} \times \Aut\T
\\[5pt]
& & f & \longmapsto & (f_u)_{u\in\lev_{\ell}},
\end{matrix}
\]
where the entries of the tuple $(f_u)_{u\in\lev_{\ell}}$ are ordered according to the lexicographic order in $\lev_{\ell}$ derived from the ordering $x_1<\cdots<x_m$ of $X$. 
For simplicity, we write $\psi$ for $\psi_1$.
An automorphism $f$ of $\T$ is \emph{rooted} if it permutes rigidly the subtrees hanging from the vertices in the first level of $\T$.
In other words, if $x\in X$ and $z\in X^*$  then $f(xz)=\rho(x)z$ for some $\rho\in \sym(X)$.
We then say that $f$ is the rooted automorphism corresponding to the permutation $\rho$.

Now let $G$ be a subgroup of $\Aut\T$.
We say that $G$ is \emph{spherically transitive} if it acts transitively on every level $\lev_{\ell}$, and that $G$ is \emph{self-similar} if $g_u\in G$ for every $g\in G$ and $u\in X^*$.
The \emph{$\ell$th level stabilizer} of $G$ is $\st_G(\ell):=\st(\ell)\cap G$.
For a vertex $u$ of $\T$, we write $\st_G(u)$ for the \emph{stabilizer of $u$} in $G$, and 
$\rst_G(u)$ for the \emph{rigid stabilizer of $u$} in $G$, that is, the subgroup consisting of the automorphisms in $G$ that stabilize all vertices that are not descendants of $u$.
Note that $\rst_G(u)\le \st_G(u)$.
Then the \emph{$\ell$th rigid stabilizer} of $G$ is
\[
\rst_G(\ell) := \langle  \rst_G(u) \mid u\in  \lev_{\ell} \rangle
= \prod_{u\in\lev_{\ell}} \, \rst_G(u).
\]
If $G$ is spherically transitive, we say that $G$ is \emph{weakly branch} if
$\rst_G(\ell)\neq 1$ for all $\ell$, and that $G$ is \emph{branch} if $|G:\rst_G(\ell)|$ is finite for all~$\ell$.
On the other hand, if $G$ is spherically transitive and self-similar, and for some
$1\ne K\le \st_G(1)$ we have $K\times \overset{m}{\cdots} \times K\subseteq \psi(K)$, we say that
$G$ is \emph{weakly regular branch} over $K$.
If furthermore $|G:K|$ is finite, we say that $G$ is \emph{regular  branch} over $K$.
It is easy to see that (weakly) regular branch groups are also (weakly) branch.

We can now introduce GGS-groups. 
Given a non-zero tuple $\mathbf{e} =(e_1,\ldots,e_{m-1})$ in $(\Z/m\Z)^{m-1}$, the
\emph{GGS-group} $G$ corresponding to the \emph{defining vector}
$\mathbf{e}$ is the subgroup $G=\langle a,b\rangle$ of $\Aut\T$, where $a$ denotes  the  rooted  automorphism corresponding to the permutation $\sigma=(x_1,x_2,\ldots,x_m)$ and $b\in\st(1)$ is defined recursively by the condition
\[
\psi(b)=(a^{e_1},\ldots,a^{e_{m-1}},b).
\]
Note that $a$ is of order $m$ and that $b$ is of order $m/d$, where
$d=\gcd(e_1,\ldots,e_{m-1},m)$.
It is obvious that $\mathbf{e}$ and $\lambda\mathbf{e}$ define the same GGS-group if
$\lambda$ is invertible modulo $m$.

Throughout the paper, for a given prime $p$ and $n\in\N$, we let $G$ denote a GGS-group defined over the $p^n$-adic tree, having canonical generators $a$ and $b$, and defining vector
$\mathbf{e}=(e_1,\ldots,e_{p^n-1})$.
Our aim is to study whether $G$ is a (regular) weakly branch or branch group.
This problem was first addressed by Vovkivsky, who proved that the following three conditions are equivalent when $G$ is periodic (see \cite[Theorem 3]{Vovkivsky2000}):
\begin{enumerate}
\item[(a)]
There exists $i\in \{1,\ldots,p^n-1\}$ such that $e_i\not\equiv 0\bmod p$.
\item[(b)]
The group $G$ is regular branch over $G''$.
\item[(c)]
The group $G$ is branch.
\end{enumerate}
He also showed that $G$ is periodic if and only if
\begin{equation}
\label{eq: periodicity criterion}
S_i:=e_{p^i}+e_{2p^i}+\cdots+e_{p^n-p^i}\equiv 0
\hspace{-6pt}
\pmod{p^{i+1}},
\
\text{for every $i=0,\ldots,n-1$.}
\end{equation}
Thus the defining vectors for periodic GGS-groups can be obtained by arbitrarily choosing all
entries $e_j$ where $j$ is not a power of $p$, and then using the conditions in
(\ref{eq: periodicity criterion}) to determine $e_{p^i}$ modulo $p^{i+1}$ for every $i=0,\ldots,n-1$
(actually in reverse order of these values).
It follows that the proportion of vectors defining periodic GGS-groups is roughly
\[
\frac{1}{p} \cdot \frac{1}{p^2} \cdots \frac{1}{p^n} = \frac{1}{p^{n(n+1)/2}},
\]
and as a consequence, Vovkivsky's criterion does not apply to a majority of GGS-groups.
On the other hand, in the special case of the $p$-adic tree, Fern\'andez-Alcober and Zugadi-Reizabal \cite{Fernandez-Alcober2014} proved that all GGS-groups with non-constant defining vector are regular branch over either $G'$ or $\gamma_3(G)$.
Later on, Fern\'andez-Alcober, Garrido and Uria-Albizuri \cite{Fernandez-Alcober2017} showed that the group with constant defining vector is not branch, although it is weakly regular branch.
This completes the analysis of the branch properties of GGS-groups over the $p$-adic tree.
The goal of the present paper is to extend Vovkivsky's results and the results regarding the $p$-adic tree to non-periodic groups and to general primary trees, respectively.

Before stating our main theorem we need to introduce some notation.
Let
\[
\F(p^n)=(\Z/p^n\Z)^{p^n-1}\setminus(p\Z/p^n\Z)^{p^n-1}
\]
be the set of defining vectors that are not trivial modulo $p$.
Observe that Vovkivsky's criterion says that a periodic GGS-group over a primary tree is branch if and only if $\mathbf{e}\in\F(p^n)$.
We will prove (see Lemma \ref{lem: e in F}) that a GGS-group with $\mathbf{e}\not\in \F(p^n)$ is not spherically transitive, so it cannot be (weakly) branch.
Hence our study reduces to GGS-groups with defining vector in $\F(p^n)$. 
Given $\mathbf{e}\in\F(p^n)$, let
\begin{equation}
\label{def: Y}
Y(\mathbf{e}) := \{ 1\le i\le p^n-1 \mid e_i\not\equiv 0\bmod p \}
\end{equation}
be the set of indices having invertible entries in $\mathbf{e}$, and let
\begin{equation}
\label{def: t}
t(\mathbf{e}) := \max \{ s\in \Z \mid s\geq 0 \text{ and } p^s
\mid i \text{ for all } i\in Y(\mathbf{e}) \}.
\end{equation}
If there is no confusion about the vector $\mathbf{e}$ then we simply write $Y$ and $t$ for $Y(\mathbf{e})$ and $t(\mathbf{e})$, respectively.
Note that $t\in\{ 0,1,\ldots,n-1 \}$ and $Y\subseteq \{p^t,2p^t,\ldots,p^n-p^t\}$.
We say that $Y$ is \emph{maximal} if $Y=\{p^t,2p^t,\ldots,p^n-p^t\}$.

Now we define two special subsets of $\F(p^n)$.
On the one hand, let $\E(p^n)$ be the set of all tuples $\mathbf{e}$ that are constant modulo $p$ on the set $Y$, and such that $Y$ is maximal.
This is equivalent to the condition $e_{ip^t}\equiv e_{jp^t}\bmod p$ for $1\le i,j\le p^{n-t}-1$.
On the other hand, we define $\E'(2^n)=\{ \mathbf{e}\in\F(2^n) \mid t=n-1 \}$; in other words,
$\E'(2^n)$ consists of the vectors whose only odd entry is $e_{2^{n-1}}$.
Obviously, we have $\E'(2^n)\subseteq \E(2^n)$.
Then our first main theorem reads as follows.

\begin{thm}
\label{thm: main 1}
Let $G$ be a GGS-group over the $p^n$-adic tree with defining vector $\mathbf{e}\in\F(p^n)$.
The following hold:
\begin{enumerate}
\item
If $\mathbf{e}\not\in\E'(2^n)$ then $G$ is weakly regular branch over $G''$.
\vspace{3pt}
\item
If $\mathbf{e}\not\in\E(p^n)$, then $G$ is regular branch over $\gamma_3(G)$.
\end{enumerate}
\end{thm}

Observe that (i) of Theorem~\ref{thm: main 1} extends Vovkivsky's result to practically all GGS-groups; in fact, to all GGS-groups if $p$ is odd.
On the other hand, part (ii) applies to most defining vectors in $\F(p^n)$, since
\[
|\F(p^n)| = (p^{\,p^n-1}-1) p^{(n-1)(p^n-1)}
\]
and
\[
|\E(p^n)| =n(p-1)p^{(n-1)(p^n-1)}.
\]
It also applies to periodic GGS-groups (see Corollary~\ref{cor:extend vovkivsky}), and consequently it improves Vovkivsky's result by showing that periodic GGS-groups are regular branch over a larger subgroup than $G''$.
Finally, we want to remark that in some of the cases in (ii) we actually get the better result that $G$ is regular branch over $G'$.

The question of which GGS-groups with $\mathbf{e}\in\E'(2^n)$ are weakly branch is still open.
Observe that for $n=1$ we simply have $\E'(2)=\F(2)=\{(e_1)\mid \text{$e_1$ is odd}\}$.
These vectors define just one GGS-group, which is isomorphic to the infinite dihedral group and is not weakly branch.

On the other hand, it is far from clear when a GGS-group with defining vector in $\E(p^n)$ is 
branch or simply weakly branch.
For some special defining vectors, we can show that the corresponding groups are also regular branch over $\gamma_3(G)$ (see Theorem~\ref{thm:partially constant}).
However, if the defining vector is constant, we have the following result.

\begin{thm}
\label{thm: main 2}
Let $G$ be a GGS-group over the $p^n$-adic tree with constant defining vector.
Then $G$ is not a branch group.
\end{thm}

Finally, observe that, in the case of the $p$-adic tree, $\F(p)$ consists of all non-zero vectors and
$\E(p)$ reduces to the constant non-zero vectors.
Consequently, Theorems~\ref{thm: main 1} and \ref{thm: main 2} generalize the above-mentioned results of Fern\'andez-Alcober and Zugadi-Reizabal \cite{Fernandez-Alcober2014}, and of
Fern\'andez-Alcober, Garrido, and Uria-Albizuri \cite{Fernandez-Alcober2017}, respectively.

\medskip

\noindent
\textit{Notation.}
If $\sigma$ and $\tau$ are two permutations of a set, we write their composition (where we apply first $\sigma$ and then $\tau$) as $\sigma\tau$, by juxtaposition, rather than $\tau\circ\sigma$.
This applies in particular to the composition in $\Aut\T$.
On the other hand, when we write dots in a tuple in between two entries equal to $1$, like in
$(1,\ldots,1,g)$, that means that all entries represented by the dots are also equal to $1$.
However, if the dots are not between two entries equal to $1$, like in
$(g,\dots,1)$, then they represent unspecified elements.

\medskip

\noindent
\textbf{Acknowledgement.}
The first author thanks A.\ Caranti, O.\ Garaialde Oca\~na and J.\ Gonz\'alez S\'anchez  for helpful discussions.

\medskip

\noindent
\textbf{Data availability statement.}
Data sharing not applicable to this article as no datasets were generated or analysed during the current study.

\section{Regular branch GGS-groups}
\label{sec:reg branch}

Recall that $G$ always denotes a GGS-group over the $p^n$-adic tree, having canonical generators $a$ and $b$, and defining vector $\mathbf{e}=(e_1,\ldots,e_{p^n-1})$.
Following the notation in \cite{Vovkivsky2000}, we let $p^{R_0}$ be the highest power of $p$ dividing all entries of $\mathbf{e}$.
Hence $o(b)=p^{n-R_0}$ and the set $\F(p^n)$ consists of all defining vectors for which $R_0=0$.
Also we set $b_i:=b^{a^i}$ for every integer $i$.
Then
\begin{equation}
\label{eqn:stG(1)}
\begin{split}
\psi(b_0) &= (a^{e_1},\ldots,a^{e_{p^n-1}},b),
\\
\psi(b_1) &= (b,a^{e_1},\ldots,a^{e_{p^n-1}}),
\\
&\,\,\,\vdots
\\
\psi(b_{p^n-1}) &= (a^{e_2},\ldots, a^{e_{p^n-1}},b,a^{e_1}),
\end{split}
\end{equation}
and $b_i=b_j$ if $i\equiv j\pmod{p^n}$.
Since $G/\langle b \rangle^G$ is generated by the image of $a$, and $a^{p^i}\in\st_G(1)$ if and only if $i\ge n$, it readily follows that
\begin{equation}
\label{eqn:stG(1)}
\st_G(1) = \langle b \rangle^G = \langle b_0, b_1, \ldots, b_{p^n-1} \rangle
\end{equation}
has index $p^n$ in $G$.
Similarly, we have $\st_G(1)=\st_G(x_i)$ for every $x_i\in X$.

Our first theorem gives the structure of the abelianization of $G$.
This can be accomplished by using a result of Rozhkov \cite[Proposition 1]{rozhkov1986theory}, and the proof follows exactly the same strategy as in Propositions 4.2 and 4.3 of
\cite{alexoudas2016maximal}, where Alexoudas, Klopsch, and Thillaisundaram determine the abelianization of multi-edge spinal groups.
For this reason, we omit the details of the proof of the theorem, and refer the reader to the latter paper.
Strictly speaking, Rozhkov's result applies to the so-called Aleshin type groups (AT-groups), which are  spherically transitive by definition.
GGS-groups over a primary tree are not necessarily spherically transitive (see Lemma \ref{lem: e in F} below); however, a careful analysis of the proof of Proposition 1 of \cite{rozhkov1986theory} shows that this transitivity is not actually needed, and the result also applies in our setting.
We thus have the following theorem.

\begin{theorem}
\label{thm:abelianisation}
Let $G$ be a GGS-group over the $p^n$-adic tree.
Then
$G/G'=\langle aG' \rangle \times \langle bG' \rangle \cong C_{p^n}\times C_{p^{n-R_0}}$.
\end{theorem}

Note also that all terms of the lower central series have finite index in $G$, since $G$ can be generated by two elements of order $p$.

It readily follows from the definition that all GGS-groups are self-similar.
Consequently we can consider the group homomorphism $\psi_u:\st_G(u)\rightarrow G$ given by $g\mapsto g_u$.
Recall that a subgroup $G$ of $\Aut\T$ is said to be \emph{fractal} if it is self-similar and
$\psi_u$ is onto.
We have the following result.

\begin{lem}
\label{lem: e in F}
Let $G$ be a GGS-group over the $p^n$-adic tree with defining vector $\mathbf{e}$.
Then the following conditions are equivalent:
\begin{enumerate}
\item
$G$ is spherically transitive.
\item
$G$ is fractal.
\item
$\mathbf{e}\in\F(p^n)$.
\end{enumerate}
\end{lem}

\begin{proof}
We prove that $\mathbf{e}\in\F(p^n)$ implies that $G$ is spherically transitive and 
fractal, and that $\mathbf{e}\not\in\F(p^n)$ implies that $G$ is neither spherically transitive nor fractal.

We first assume that $\mathbf{e}\in\F(p^n)$.
By \cite[Lemma 2.7]{Uria2016}, in order to prove that $G$ is spherically transitive and fractal, it suffices to see that $G$ acts transitively on the vertices of the first level of $\T$ and that
$\psi_{x_i}(\st_G(x_i))=G$ for some $x_i\in X$.
The former is obvious, since $a\in G$, and for the latter observe that since $\mathbf{e}\in\F(p^n)$ we have $p\nmid e_i$ for some $i$ and then $\psi_{x_p}(b_{-i})=a^{e_i}$ and $\psi_{x_p}(b)=b$ generate $G$.

Now let $\mathbf{e}\not\in\F(p^n)$.
Since $p$ divides all components of $\mathbf{e}$, for every $x_i\in X$ and $g\in\st_G(x_i)$ the section $g_{x_i}$ is a word in $\{a^p,b\}$.
Thus $G$ is not fractal.
Now assume for a contradiction that $G$ is spherically transitive.
Then there exists $g\in G$ such that $g(x_1x_1)=x_1x_2$.
However, by the above we have $g(x_1x_1)=g(x_1)g_{x_1}(x_1)=x_1x_j$ for some
$j\equiv 1\pmod p$, which is a contradiction.
\end{proof}

As a consequence, since (weakly) branch groups are spherically transitive by definition, in the remainder we will always assume that $\mathbf{e}\in\F(p^n)$, unless otherwise stated.
Then $R_0=0$ and both $a$ and $b$ have order $p^n$.

The next lemma is one of the main tools for finding a branch structure in a GGS-group and it  can be proved as in \cite[Proposition~2.18]{Fernandez-Alcober2014}.

\begin{lem}
\label{lem: main tool for branchness}
Let $G$ be a spherically transitive fractal subgroup of $\Aut\T$, where $\T$ is a regular rooted tree, and let $L$ and $N$ be two normal subgroups of $G$.
Suppose that $L=\langle S \rangle^G$ and that $(1,\ldots,1,s,1,\ldots,1)\in\psi(N)$ for every $s\in S$, where $s$ appears always at the same position in the tuple.
Then $L\times \cdots \times L \subseteq \psi(N)$.
\end{lem}

The next lemma generalizes \cite[Theorem~2.16]{Fernandez-Alcober2014}.
It shows that if a component $e_k$ of $\mathbf{e}$ is invertible modulo $p$ then there exists another GGS-group that is conjugate to $G$ in $\Aut\T$, and whose defining vector has the $p^s$-th component equal to $1$, where $p^s$ is the highest power
of $p$ dividing $k$.

\begin{lem}
\label{lem:reduced_definining_vector}
Let  $G$ be a GGS-group over the $p^n$-adic tree with defining vector
$\mathbf{e}\in \F(p^n)$, and assume that $e_k\not  \equiv  0 \bmod  p$.
If $p^s$ is the highest power of $p$ dividing $k$ then there exist $\alpha\in\sym(p^n-1)$ and $f\in\Aut\T$ such that:
\begin{enumerate}
\item
$\alpha(p^s)=k$.
\item
$\alpha(p^n-i)=p^n-\alpha(i)$ for all $i=1,\ldots,p^{n}-1$.
\item
$G^f$ is the GGS-group with defining vector
$\mathbf{e'}=e_k^{-1}\, (e_{\alpha(1)},\ldots,e_{\alpha(p^n-1)})$.
In particular, $e'_{p^s}=1$.
\end{enumerate}
\end{lem}

\begin{proof}
By hypothesis we can write $k=hp^s$ where $h\not \equiv 0\bmod p$.
If $r$ is a solution to the congruence $hr\equiv  1\bmod p^{n-s}$, then the permutation
$\delta\in\sym(p^n)$ given by $\delta(i)\equiv ri  \bmod p^n$ for every $i$ satisfies that
$\sigma^{\delta}=\sigma^r$ and $\delta(k)=p^s$.

Let us define $f\in\Aut\T$ recursively by $f=dh$, where $d$ is the rooted automorphism corresponding to $\delta$, and $h\in\st(1)$ is defined via $\psi(h)=(f,\ldots,f)$.
Note that $h$ commutes with any rooted automorphism, since its components under $\psi$ are all the same.
Then
\begin{equation}
\label{eqn:a^f}
a^f=(a^d)^h=(a^r)^h=a^r,
\end{equation}
since $a^d$ is the rooted automorphism corresponding to the permutation
$\sigma^{\delta}=\sigma^r$.

Now let $\alpha=\delta^{-1}$ and observe that $\alpha$ satisfies (i) and (ii).
Also
\[
\psi(b^f)
=
\psi(b^d)^{\psi(h)}
=
(a^{e_{\alpha(1)}},\ldots,a^{e_{\alpha(p^n-1)}},b)^{\psi(h)}
=
(a^{re_{\alpha(1)}},\ldots,a^{re_{\alpha(p^n-1)}},b^f),
\]
by using (\ref{eqn:a^f}).
Since $r\not\equiv 0\pmod p$, it follows that
$G^f=\langle a^r,b^f\rangle=\langle a,b^f \rangle$ is the GGS-group with defining vector
$(re_{\alpha(1)},\ldots,re_{\alpha(p^n-1)})$.
If we multiply this vector by the inverse of $re_k$ modulo $p^n$ then we see that $G^f$ is also the GGS-group with defining vector $\mathbf{e'}$, which proves (iii).
\end{proof}

It is easy to see that, for any subgroup $G$ of $\Aut\T$ and any $f\in\Aut(\T)$, we have
$\rst_G(u)^f=\rst_{G^f}(f(u))$ for every vertex $u$ of $\T$.
Hence $\rst_G(\ell)^f=\rst_{G^f}(\ell)$ for every $\ell\in\N$.
Thus, by the previous lemma, in order to study branch properties in a GGS-group, we may assume without loss of generality that $e_{p^t}=1$, where $t=t(\mathbf{e})$ is as in \eqref{def: t}.
In the remainder of this section, we fix the notation $k:=p^t$.

\begin{definition}
A GGS-group is \emph{invertible-symmetric}, IS for short, if the set $Y$ is symmetric, in the sense that $i$ belongs to $Y$ if and only if $p^n-i$ does.
In other words, a component $e_i$ of $\mathbf{e}$ is invertible if and only if $e_{p^n-i}$ is.
\end{definition}

We start our analysis of branch structures in GGS-groups by dealing with the case when $G$ is not IS.
We assume that $e_k=1$, and then $q:=e_{p^n-k}\equiv 0\bmod p$ by (ii) of Lemma~\ref{lem:reduced_definining_vector}.
We define a sequence $\{g_i\}_{i\ge 0}$ of automorphisms of $\T$ by means of \begin{equation}
\label{eq: g_i}
\psi(g_i) =  (1,\ldots,1,[a,b],1,\ldots,1)\cdot(1,\ldots,1,[b^{q^i},a^{q^{i+1}}],1,\ldots,1),
\end{equation}
where the non-trivial components appear in the $k$-th position and in the $(p^n-2ik)$-th position, respectively, the latter being understood modulo $p^n$.

\begin{lem} 
\label{lem:asymmetric_e} 
The sequence $\{g_i\}_{i\ge 0}$ defined in \eqref{eq: g_i} is contained in $\st_G(1)'$.
\end{lem}

\begin{proof}
Since $\psi([b,b_k])=(1,\ldots,1,[a,b],1,\ldots,1,[b,a^q])$,
where the non-trivial components are at positions $k$ at $p^n$, we have $g_0=[b,b_k]$.
Similarly,
\begin{equation}
\label{eq:comm [b,b_k]1}
\psi([b^{q^i},b_k^{q^{i-1}}]^{a^{-(2i-1)k}})
=
(1,\ldots,1,[a^{q^i},b^{q^{i-1}}],1,\ldots,1)
\cdot
(1,\ldots,1,[b^{q^i},a^{q^i}],1,\ldots,1),
\end{equation}
where the non-trivial components appear at positions $p^n-(2i-2)k$ and $p^n-(2i-1)k$,
respectively, and
\begin{equation}
\label{eq:comm [b,b_k]2}
\psi([b^{q^i},b_k^{q^{i}}]^{a^{-2ik}})
=
(1,\ldots,1,[a^{q^i},b^{q^i}],1,\ldots,1)
\cdot
(1,\ldots,1,[b^{q^i},a^{q^{i+1}}],1,\ldots,1),
\end{equation}
with non-trivial components at positions $p^n-(2i-1)k$ and $p^n-2ik$.
By combining \eqref{eq: g_i}, \eqref{eq:comm [b,b_k]1}, and \eqref{eq:comm [b,b_k]2}, one can readily check that 
\[
g_i = g_{i-1} [b^{q^i},b_k^{q^{i-1}}]^{a^{-(2i-1)k}} [b^{q^i},b_k^{q^i}]^{a^{-2ik}}
\]
for all $i\ge 1$.
Thus $g_i\in\st_G(1)'$ by induction on $i$.
 \end{proof}

The following result is a consequence of the previous lemma.

\begin{theorem}
\label{thm:e_p^n-k}
If $G$ is not IS then $\psi(\st_G(1)')=G'\times \cdots \times G'$.
In particular, $G$ is regular branch over $G'$.
\end{theorem}

\begin{proof}
The inclusion $\subseteq$ is obvious since $G$ is self-similar, so we only need to prove
$\supseteq$.
Let $g_i$ be defined as in \eqref{eq: g_i}.
Since $o(b)=p^n$ and $q$ is divisible by $p$, we have $[b^{q^n},a^{q^{n+1}}]=1$.
Thus $\psi(g_n)$ has all components  equal to $1$ with the exception of the component at
position $k$, which is equal to $[a,b]$.
Since $g_n\in\st_G(1)'$ by Lemma~\ref{lem:asymmetric_e} and $G'=\langle[a,b]\rangle^G$, the desired inclusion follows from Lemma~\ref{lem: main tool for branchness}.
\end{proof}

After having proved Theorem~\ref{thm:e_p^n-k}, we next assume that $G$ is IS.
We continue our analysis of branch structures by considering the case when $Y$ is not maximal, that is, when $Y\subsetneq \{p^t,2p^t,\ldots,p^n-p^t\}$.

Let $h$ be the smallest integer in $\{1,\ldots,p^{n-t}-1\}$ such that $hk\not\in Y$.
Note that $h\ge 2$.
Then we set $q:=e_{p^n-hk}$, $y:=e_{p^n-(h-1)k}$, and $z:=e_{p^n-k}$; in other words, $q$, $y$ and $z$ are the symmetrical components of $e_{hk}$, $e_{(h-1)k}$ and $e_k$.
Thus $p$ divides $q$, and $y$ and $z$ are invertible modulo $p$.
In this case we define a sequence $\{g_i\}_{i\ge 0}$ of automorphisms of $\T$ as follows:
\begin{equation}
\label{eq: g_i IS}
\psi(g_i)
=
(1,\ldots,1,[a,b,a],1,\ldots,1)\cdot(1,\ldots,1,[b^{z^i},a^{z^{i+1}}, a^{q^{2i+1}y^{-(2i+1)}}], 1,\ldots,1),
\end{equation}
where the non-trivial components are the $k$-th and the $(p^n-2ik)$-th.

\begin{lem}
\label{lem:sequence for IS vector}
The sequence $\{g_i\}_{i\ge 0}$ defined in \eqref{eq: g_i IS} is contained in
$\gamma_3(\st_G(1))$.
\end{lem}

\begin{proof}
It is easy to see that $g_0=[b,b_k,b_{hk}^{y^{-1}}]$.
We claim that $g_i=g_{i-1}c_1^{-1}c_2$, where
\[
c_1
=
[b_k^{z^{i-1}},b^{z^i},b_{hk}^{q^{2i-1}y^{-2i}}]^{a^{-(2i-1)k}}
\]
and
\[
c_2
=
[b^{z^i},b_k^{z^i},b_{hk}^{q^{2i}y^{-(2i+1)}}]^{a^{-2ik}}.
\]
Then $g_i$ belongs to $\gamma_3(\st_G(1))$ by induction on $i$.

The claim follows immediately from \eqref{eq: g_i IS} by taking into account that
\begin{multline*}
\psi(c_1)
=
(1,\ldots,1,[b^{z^{i-1}},a^{z^i},a^{q^{2i-1}y^{-(2i-1)}}],1,\ldots,1)
\\
\cdot
(1,\ldots,1,[a^{z^{i}},b^{z^i},a^{q^{2i}y^{-2i}}],1,\ldots,1),
\end{multline*}
where the non-trivial components appear in positions $p^n-(2i-2)k$ and $p^n-(2i-1)k$, and that
\begin{multline*}
\psi(c_2)
=
(1,\ldots,1,[a^{z^i},b^{z^{i}},a^{q^{2i}y^{-2i}}],1,\ldots,1)
\\
\cdot
(1,\ldots,1,[b^{z^i},a^{z^{i+1}},a^{q^{2i+1}y^{-(2i+1)}}],1,\ldots,1),
\end{multline*}
with non-trivial components $p^n-(2i-1)k$ and $p^n-2ik$.
\end{proof}

\begin{theorem}
\label{thm:IS_with_zeros}
If $G$ is IS and $Y$ is not maximal, then
$\psi(\gamma_3(\st_G(1)))=\gamma_3(G)\times \cdots \times \gamma_3(G)$.
In particular, $G$ is regular branch over $\gamma_3(G)$. 
\end{theorem}

\begin{proof}
Since $p$ divides $q$ and $a$ has order $p^n$, for large enough $i$ we have
\[ 
 \psi(g_i)=(1,\ldots,1,[a,b,a],1,\ldots,1),
\]
where $[a,b,a]$ appears in the $k$-th component.
Also $g_i\in\gamma_3(\st_G(1)))$ by Lemma~\ref{lem:sequence for IS vector}.
Moreover, since $z$ is invertible modulo $p$,
\[
\psi(b_k^{z^{-1}}b_{-k}^{-1})=(\ldots,b^{z^{-1}}a^{-e_{2k}},\ldots,1)
\]
where the first displayed component is the $k$-th one.
Hence
\[
\psi([b,b_k,b_k^{z^{-1}}b_{-k}^{-1}])=(1,\ldots,1,[a,b,b^{z^{-1}}a^{-e_{2k}}],1,\ldots,1)
\]
where the non-trivial component appears in position $k$.
Since $G=\langle a,b^{z^{-1}}a^{-e_{2k}}\rangle$ then 
\[
\gamma_3(G)=\langle[ a,b,a],[a,b,b^{z^{-1}}a^{-e_{2k}}] \rangle^G,
\]
and the result follows from Lemma~\ref{lem: main tool for branchness}.
\end{proof}

We now consider the case when $Y$ is maximal.
In this case $G$ is trivially IS.

\begin{theorem}
Suppose that $Y$ is maximal.
If there exists $m\in Y\setminus\{k,p^n-k\}$ such that
\[
\delta_m:=
\det
\begin{pmatrix}
e_{m-k} & e_m
\\
e_m & e_{m+k}
\end{pmatrix}
\not\equiv 0\bmod p
\]
then $\psi(\gamma_3(\st_G(1)))=\gamma_3(G)\times \cdots \times \gamma_3(G)$.
In particular, $G$ is regular branch over $\gamma_{3}(G)$.
\end{theorem}

\begin{proof}
In the following formulas the displayed components are the $k$-th and the last one.
We observe that
\[
\psi(b_{-m}^{e_{m-k}}b_{-m+k}^{-e_{m}})
=
(\ldots,a^{\delta_m},\ldots,  1).
\]
Since by hypothesis $\delta_m$ is invertible modulo $p^n$, there exists $g\in\st_G(1)$ such that
$\psi(g)=(\ldots,a,\ldots,1)$.
On the other hand,
\[
\psi(b_kb_{-k}^{-e_{p^n-k}})
=
(\ldots,ba^{-e_{2k}e_{p^n-k}},\ldots,1)
\]
and so by multiplying $b_kb_{-k}^{-e_{p^n-k}}$ by a suitable power of $g$ we can find $h\in\st_{G}(1)$ such that $\psi(h)=(\ldots,b,\ldots,1)$.
Consequently
\[
\psi([b,b_{k},g])=(1,\ldots,1,[a,b,a],1,\ldots, 1),
\qquad
\psi([b,b_{k},h])=(1,\ldots,1,[a,b,b],1,\ldots,  1),
\]
and $\psi(\gamma_3(\st_G(1)))=\gamma_3(G)\times \cdots \times \gamma_3(G)$
by Lemma~\ref{lem: main tool for branchness}.
\end{proof}

\begin{theorem}
Suppose that $Y$ is maximal and that for all $m\in Y\setminus\{k,p^n-k\}$ we have
\begin{equation}
\label{delta_m is 0}
\delta_m
:=
\det
\begin{pmatrix}
e_{m-k} & e_m
\\
e_m & e_{m+k}
\end{pmatrix}
\equiv 0\bmod p.
\end{equation}
If $\mathbf{e}\not\in\E(p^n)$ then
$\psi(\gamma_3(\st_G(1)))=\gamma_3(G)\times \cdots \times \gamma_3(G)$.
In particular, $G$ is regular branch over $\gamma_3(G)$.
\end{theorem}

\begin{proof}
We first deal with the case where the condition on $\delta_m$ in the statement holds vacuously.
This happens if $Y\subseteq \{k,p^n-k\}$, or equivalently, if $|Y|\le 2$.
It follows that $k=p^{n-1}$ and $p=2$ or $p=3$.
Since $\mathbf{e}\not\in\E(p^n)$, we necessarily have $p=3$ and $x:=e_{3^n-k}\equiv 2 \pmod 3$.
Let $y$ be the inverse of $x$ modulo $3^n$, so that $y\equiv 2 \pmod 3$.
Then we have
\[
\psi([b_{2k},b,bb_k^{-x}]) = (1,\ldots,1,[a,b,ba^{-x^2}])
\]
and
\[
\psi([b_k^y,b,b_{2k}^yb^{-1}]) = (1,\ldots,1,[a,b,a^yb^{-1}]).
\]
Now observe that $\langle ba^{-x^2},a^yb^{-1} \rangle=\langle ba^{-x^2},a^{y-x^2} \rangle$ coincides with $G$, since $y-x^2\equiv 1\pmod 3$.
This proves that $\psi(\gamma_3(\st_G(1)))=\gamma_3(G)\times \cdots \times \gamma_3(G)$ in this case.

Next we assume that $|Y|\ge 3$.
From (\ref{delta_m is 0}) we get
\begin{equation}
\label{eq: e_ik=e_2ik}
e_{ik}\equiv e_{2k}^{i-1}\bmod p
\end{equation}
for all $i=2,\ldots,p^{n-t}-1$.
Since $\mathbf{e}$ is not constant modulo $p$ for the indices in $Y$, it follows that
$e_{2k}\not\equiv 1\bmod p$.

We observe that
\begin{equation}
\label{eq:comm3}
\psi([b,b_k,b_k b_{-k}^{-e_{p^n-k}}])
=
(1,\ldots,1,[a,b,ba^{-e_{2k}e_{p^n-k}}],1,\ldots,1),
\end{equation}
where the displayed component is the $k$-th one.

Now set $g:=b^{e_{p^n-3k}}b_{2k}^{-e_{p^n-k}}$.
Note that the condition $|Y|\ge 3$ implies that $p^n>3k$, and so $e_{p^n-3k}$ is a well defined entry
of $\mathbf{e}$. 
Then
\[
\psi(g)
=
(\ldots, a^{e_{p^n-3k}-e_{p^n-k}^2}, \ldots, 1, \ldots),
\]
where the displayed components are the $k$-th one and the $(p^n-k)$-th one.
If $y$ is the inverse of $e_{2k}$ modulo $p^n$, we get
\begin{equation}
\label{eq:comm4}
\psi([b_{-k}^y,b_k,g])
=
(1,\ldots,1,[a,b,a^{e_{p^n-3k}-e_{p^n-k}^2}],1,\ldots,1),
\end{equation}
where the only non-trivial component is the $k$-th one.
  
Let $s:=p^{n-t}$.
By \eqref{eq: e_ik=e_2ik} and Fermat's Little Theorem, we have
\[
e_{p^n-3k}-e_{p^n-k}^2
=
e_{(s-3)k}-e_{(s-1)k}^2
\equiv
e_{2k}^{s-4}-e_{2k}^{2(s-2)}
\equiv
e_{2k}^{-3}(1-e_{2k})
\not\equiv
0 \bmod p,
\]
since $e_{2k}\not\equiv 1\pmod p$.
Hence $\{a^{e_{p^n-3k}-e_{p^n-k}^2},ba^{-e_{2k}e_{p^n-k}}\}$ is a generating set of $G$, and the result follows from \eqref{eq:comm3} and \eqref{eq:comm4}.
\end{proof}

After these preliminary results we can now prove Theorem~\ref{thm: main 1}.

\begin{proof}[Proof of Theorem~\ref{thm: main 1}]
We first prove (ii).
So we assume that $\mathbf{e}\not\in\E(p^n)$.
By the previous theorems in this section, $G$ is regular branch over $G'$ or
$\gamma_3(G)$.
If $G$ is regular branch over $G'$, we know that $\psi(\st_G(1)')=G'\times \cdots \times G'$.
Since the stabilizer of every vertex in the first level coincides with $\st_G(1)$ and $G$ is fractal by Lemma \ref{lem: e in F}, it follows that $\psi(\st_G(1))$ is a subdirect product of
$G\times \cdots \times G$.
It immediately follows that
$\psi(\gamma_3(\st_G(1)))=\gamma_3(G) \times \cdots \times \gamma_3(G)$, and $G$ is also regular branch over $\gamma_3(G)$.
This completes the proof of (ii).

Now we assume that $\mathbf{e}\not\in\E'(2^n)$ and prove (i).
If $G$ is not IS, then it is regular branch over $G'$ by Theorem~\ref{thm:e_p^n-k}.
Consequently
\[
\psi(\st_G(1)'') = G''\times \cdots \times G'',
\]
and $G$ is weakly regular branch over $G''$.
Let us now suppose that $G$ is IS.
For every $g_1,g_2\in G$ there exist $h_1,h_2\in\st_G(1)$ such that the $k$-th components of $\psi(h_1)$ and $\psi(h_2)$ are $g_1$ and $g_2$, respectively.
Then
\[
\psi([[b,b_k]^{h_1},[b_k,b_{2k}]^{h_2}])
=
(1,\ldots,1,[[a,b]^{g_1},[b,a^{e_{p^n-k}}]^{g_2}],1,\ldots,1),
\]
where the non-trivial component is at the $k$-th position.
Here we need to use that $2k\ne p^n$, since $\mathbf{e}\not\in\E'(2^n)$.
Now since $G$ is IS, the component $e_{p^n-k}$ is not divisible by $p$.
Thus the sets $\{[a,b]^{g_1}\mid g_1\in G\}$ and $\{[b,a^{e_{p^n-k}}]^{g_2}\mid g_2\in G\}$ are generating sets for $G'$.
Then $G''=\langle[[a,b]^{g_1},[b,a^{e_{p^n-k}}]^{g_2}]\mid g_1,g_2\in G\rangle^G$ and thus  $G$ is weakly regular branch over $G''$ by Lemma~\ref{lem: main tool for branchness}.
\end{proof}

We can deduce an improved version of Vovkivsky's result from
Theorem~\ref{thm: main 1}.

\begin{cor}
\label{cor:extend vovkivsky}
Let $G$ be a periodic GGS-group over the $p^n$-adic tree.
Then $G$ is regular branch over $\gamma_3(G)$.
\end{cor}

\begin{proof}
It suffices to show that $\mathbf{e}\not\in\E(p^n)$.
Otherwise all $e_{ip^t}$ have the same (non-zero) value modulo $p$, for
$i=1,\ldots,p^{n-t}-1$.
Now since $G$ is periodic, from \eqref{eq: periodicity criterion} we get
\[
S_t = e_{p^t} + e_{2p^t} + \cdots + e_{p^n-p^t} = (p^{n-t}-1)e_{p^t} \equiv 0 \pmod p.
\]
This implies that $e_{p^t}\equiv 0\pmod p$, which is a contradiction.
\end{proof}

We close this section by showing that $G$ is also regular branch over $\gamma_3(G)$ in some special cases with $\mathbf{e}\in\mathcal{E}(p^n)$.

\begin{theorem}
\label{thm:partially constant}
Suppose that $\mathbf{e}$ is constant on $Y$ and constant equal to $0$ outside $Y$, without $\mathbf{e}$ being constant.
Then $G$ is a regular branch group over $\gamma_3(G)$.
\end{theorem}

\begin{proof}
We have
\begin{equation}
\label{eqn:partially constant}
\psi([b,a])=(b,1,\ldots,1,a^{-1},a,1,\ldots,1,a^{-1},a,1,\ldots,1,b^{-1}),
\end{equation}
where the components equal to $a^{-1}$ are those in positions multiple of $k$.
From the assumptions about $\mathbf{e}$, we have $e_1=0$, and by taking the commutator of
(\ref{eqn:partially constant}) with $\psi(b)$, it follows that $[b,a,b]=1$.
Hence $\gamma_3(G)=\langle [a,b,a] \rangle^G$.
Now the result follows from
\[
\psi([b,b_k,b]) =(1,\ldots,1,[a,b,a],1,\ldots,1,[b,a,b]) = (1,\ldots,1,[a,b,a],1,\ldots,1).
\]
\end{proof}

\section{GGS-groups with constant defining vector}
\label{sec:constant GGS}

In this section we prove that the GGS-groups with constant defining vector are not branch, as in the case of the $p$-adic tree.
Note that, by Theorem~\ref{thm: main 1}, they are weakly branch unless $p^n=2$.

In the following, $\G$ denotes the GGS-group defined on the $p^n$-adic tree by the vector
$\mathbf{e}=(1,\ldots,1)$.
We introduce the subgroup $K=\langle ba^{-1} \rangle^{\G}$ of $\G$, and we set
$y_0=ba^{-1}$ and $y_i=y_0^{a^i}$ for all $i\in\Z$.
Then $y_i=y_j$ if $i\equiv j \pmod{p^n}$.
One can easily check that
\[
y_{p^n-1} y_{p^n-2} \cdots y_1 y_0 = 1.
\]
Also $\G/K=\langle aK \rangle=\langle bK \rangle$, and consequently $\G'\le K$ and $|\G:K|=p^n$.
The following lemma generalizes \cite[Lemma~4.2]{Fernandez-Alcober2014} to the case $n>1$, and it can be proved similarly.

\begin{lem}
\label{lem: gen properties of GGS constant}
The following hold: 
\begin{itemize}
\item[(i)]
$K=\langle y_{0},\ldots,y_{p^n-1}\rangle$.
\item[(ii)]
$K'\times\cdots\times K'\subseteq\psi(K')\subseteq \psi(\G')\subseteq K\times\cdots\times K$.
In particular $\G$ is weakly regular branch over $K'$.
\end{itemize}
\end{lem}

\begin{pro}
\label{pro:stG(1)'=stG(2)}
We have $\st_{\G}(1)'=\st_{\G}(2)$.
\end{pro}

\begin{proof}
From (\ref{eqn:stG(1)}) we get
$\st_{\G}(1)'=\langle [b_i,b_j] \mid i,j=1,\ldots,p^n \rangle^{\G}\le \st_{\G}(2)$.
For the reverse inclusion, let $g\in\st_{\G}(2)$ be arbitrary.
We can write $g=hg'$, where $g'\in \st_{G}(1)'$ and $h$ is of the form
\[
h = b_1^{k_1} b_2^{k_2} \cdots b_{p^n}^{k_{p^n}},
\]
for some integers $k_i$.
Observe that $\psi(h)$ is given by the vector
\[
( b^{k_1}a^{k_2+\cdots+k_{p^n}},
a^{k_1}b^{k_2}a^{k_3+\cdots+k_{p^n}},
\ldots,
a^{k_1+\cdots+k_{p^n-2}}b^{k_{p^{n-1}}}a^{k_{p^n}},
a^{k_1+\cdots+k_{p^n-1}}b^{k_{p^n}} ).
\]
Since $h=g(g')^{-1}\in\st_{\G}(2)$ and $a$ has order $p^n$ modulo $\st_{\G}(1)$, if we set $S=k_1+\cdots+k_{p^n}$, then $S-k_i \equiv 0 \mod p^n$ for $i=1,\ldots,p^n$.
From these congruences it readily follows that $k_i\equiv 0 \mod p^n$ for all $i$, and consequently $g=g'\in\st_{\G}(1)'$, as desired.
\end{proof}

The following result generalizes \cite[Lemma~4.4]{Fernandez-Alcober2014}.

\begin{pro}
\label{pro:prod of g_i}
Let $g\in \st_{\G}(1)$ and write $\psi(g)=(g_1,\ldots,g_{p^{n}})$.
Then the following hold:
\begin{enumerate}
\item[(i)]
If $g\in \G'$ then $\prod_{i=1}^{p^n} \, g_i \in K'$.
\vspace{5pt}
\item[(ii)]
If $g\in K'$ then $\prod_{i=1}^{p^n-1} \, g_i^a g_i^{a^{2}} \cdots g_i^{a^i}\in K'$.
\end{enumerate}
\end{pro}

\begin{proof}
(i)
For every $h\in\st_{\G}(1)$ we define $\pi(h)=\prod_{i=1}^{p^n} \, h_i$, where
$\psi(h)=(h_1,\ldots,h_{p^n})$.
Since $\st_{\G}(1) = \langle b_i \mid i=1,\ldots,p^n \rangle$ and
\[
\pi(b_i) = a^{i-1} b a^{p^n-i} = (ba^{-1})^{a^{-i+1}} = y_{-i+1} \in K,
\]
it follows that $\pi(\st_{\G}(1))\subseteq K$.
Then the map $\overline\pi:\st_{\G}(1)\rightarrow K/K'$ given by $\overline\pi(h)=\pi(h)K'$ is a group homomorphism, and since $\ker\overline\pi$ is clearly invariant under conjugation by $a$, we have
$\ker\overline\pi\trianglelefteq \G$.
Now observe that $\psi([a,b])=(b^{-1}a,1,\ldots,1,a^{-1}b)$ implies that $[a,b]\in\ker\pi$.
Hence $\G'=\langle [a,b] \rangle^{\G}\le \ker\overline\pi$ and (i) follows.

(ii)
This can be proved exactly as in \cite[Lemma~4.4]{Fernandez-Alcober2014}.
\end{proof}

\begin{cor}
\label{cor:components in K'st_G(l-1)}
Let $g\in K'\st_{\G}(\ell)$ for some $\ell\in\mathbb{N}$.
If $\psi(g)=(x,1,\ldots,1,y)$ then both $x$ and $y$ lie in $K'\st_{\G}(\ell-1)$.
\end{cor}

\begin{proof}
Part (i) of Proposition~\ref{pro:prod of g_i} implies that $xy\in K'\st_{\G}(\ell-1)$, and part (ii) that
$x^a\in K'\st_{\G}(\ell-1)$.
Thus $x,y\in K'\st_{\G}(\ell-1)$.
\end{proof}

\begin{lem}
\label{lem:G/Kstab l}
For every $\ell\ge 2$ the quotient $Q_{\ell}=\G/K'\st_{\G}(\ell)$ is a $p$-group of class $\ell$ and order $p^{(\ell+1)n}$.
\end{lem}

\begin{proof}
It is obvious that $Q_{\ell}$ is a finite $p$-group, since $\G/\st_{\G}(\ell)$ is so.
The lemma will be proved if we show that $|Q_{\ell}:Q_{\ell}'|=p^{2n}$, that
$|\gamma_i(Q_{\ell}):\gamma_{i+1}(Q_{\ell})|=p^n$ for $2\le i\le \ell$, and that
$\gamma_{\ell+1}(Q_{\ell})=1$.

First of all, observe that
$Q_{\ell}/Q_{\ell}' \cong \G/\G'\st_{\G}(\ell) = \G/\G' \cong C_{p^n}\times C_{p^n}$,
by using that $\st_{\G}(2)\le \G'$ from Proposition~\ref{pro:stG(1)'=stG(2)}, and Theorem~\ref{thm:abelianisation}.
Hence
\begin{equation}
\label{eqn:exp LCS Q_l}
\exp \gamma_i(Q_{\ell})/\gamma_{i+1}(Q_{\ell}) \mid \exp Q_{\ell}/Q_{\ell}' = p^n
\end{equation}
for every $i\ge 2$.

Let us use the bar notation modulo $K'\st_{\G}(\ell)$.
Then $Q_{\ell}'=\langle [\overline b,\overline a], \gamma_3(Q_{\ell}) \rangle$.
Since $A_{\ell}=K/K'\st_{\G}(\ell)$ is an abelian normal subgroup of $Q_{\ell}$ and
$Q_{\ell}=\langle \overline a,A_{\ell} \rangle=\langle \overline b,A_{\ell} \rangle$, it follows that
\begin{equation}
\label{eqn:gamma_i(Q_l) mod next}
\gamma_i(Q_{\ell})
=
\langle [\overline b,\overline a,\overline b,\overset{i-2}{\ldots},\overline b],
\gamma_{i+1}(Q_{\ell}) \rangle
=
\langle [\overline b,\overline a,\overset{i-1}{\ldots},\overline a], \gamma_{i+1}(Q_{\ell}) \rangle
\end{equation}
for every $i\ge 2$.
From (\ref{eqn:exp LCS Q_l}) we get $|\gamma_i(Q_{\ell}):\gamma_{i+1}(Q_{\ell})|\le p^n$ for $i\ge 2$.
Hence the proof will be complete if we show that:
\begin{enumerate}
\item[(1)]
$[b,a,b,\overset{\ell-1}{\ldots},b]\in K'\st_{\G}(\ell)$.
\item[(2)]
$[b,a,b,\overset{\ell-2}{\ldots},b]^{p^{n-1}}\not\in K'\st_{\G}(\ell)$. 
\end{enumerate}
for every $\ell\ge 2$.
Indeed (1) then shows that $\gamma_{\ell+1}(Q_{\ell})=1$, while (2) shows that
$|\gamma_i(Q_{\ell}):\gamma_{i+1}(Q_{\ell})|
\ge
|\gamma_i(Q_i):\gamma_{i+1}(Q_i)|\ge p^n$, by
applying it with $i$ in the place of $\ell$.
Note that, according to (\ref{eqn:gamma_i(Q_l) mod next}), (1) is equivalent to
$[b,a,\overset{\ell}{\ldots},a]\in K'\st_{\G}(\ell)$ and (2) is equivalent to
$[b,a,\overset{\ell-1}{\ldots},a]^{p^{n-1}}\not\in K'\st_{\G}(\ell)$.

We prove (1) and (2) by induction on $\ell\ge 2$.
Suppose first that $\ell=2$.
We have
\[
\psi([b,a]) = \psi(b^{-1}b^{a})=(a^{-1}b,1,\ldots,1,b^{-1}a)
\]
and
\begin{equation}
\label{psi of [b,a,b]}
\psi([b,a,b]) = ([a^{-1}b,a],1,\ldots,1,[b^{-1}a,b]) = ([b,a],1,\ldots,1,[a,b]).
\end{equation}
The latter shows that $[b,a,b]\in\st_{\G}(2)$ and (1) holds for $m=2$.
On the other hand, if $[b,a]^{p^{n-1}}\in K'\st_{\G}(2)$ then
Corollary~\ref{cor:components in K'st_G(l-1)} implies that
$(b^{-1}a)^{p^{n-1}}\in K'\st_{\G}(1)=\st_{\G}(1)$.
Since $(b^{-1}a)^{p^{n-1}}\equiv a^{p^{n-1}} \pmod{\st_{\G}(1)}$, this is a contradiction and
(2) holds for $\ell=2$.

Now we assume that $\ell\ge 3$.
From (\ref{psi of [b,a,b]}), we get
\begin{equation}
\label{psi of [b,a,b,...,b]}
\psi([b,a,b,\overset{i}{\ldots},b])
=
( [b,a,\overset{i}{\ldots},a],1,\ldots,1,[a,b,\overset{i}{\ldots},b] )
\end{equation}
for every $i\ge 1$.
Thus, by the induction hypothesis,
\[
\psi([b,a,b,\overset{\ell-1}{\ldots},b])
\in 
(K'\st_{\G}(\ell-1)\times\cdots\times K'\st_{\G}(\ell-1))\cap \Ima\psi.
\]
Now observe that
\begin{align*}
(K'\st_{\G}(\ell-1)\times\cdots
&\times K'\st_{\G}(\ell-1))\cap \Ima\psi
\\
&=
(K'\times\cdots\times K')(\st_{\G}(\ell-1)\times\cdots\times \st_{\G}(\ell-1))\cap \Ima\psi
\\
&=
(K'\times\cdots\times K')(\st_{\G}(\ell-1)\times\cdots\times \st_{\G}(\ell-1)\cap \Ima\psi)
\\
&\subseteq
\psi(K') \psi(\st_{\G}(\ell))
=
\psi(K'\st_{\G}(\ell)),
\end{align*}
where the second equality follows from Dedekind's Law and the inclusion from
$\G$ being weakly regular branch over $K'$.
Thus $[b,a,b,\overset{\ell-1}{\ldots},b]\in K'\st_{\G}(\ell)$ and (1) holds.
Now if $[b,a,b,\overset{\ell-2}{\ldots},b]^{p^{n-1}}\in K'\st_{\G}(\ell)$ then
from (\ref{psi of [b,a,b,...,b]}) and Corollary~\ref{cor:components in K'st_G(l-1)} we get
$[b,a,\overset{\ell-2}{\ldots},a]^{p^{n-1}}\in K'\st_{\G}(\ell-1)$, contrary to the induction hypothesis.
This proves (2).
\end{proof}

Our next step is to determine the structure of the factor group $\G/K'$.

\begin{theorem}
\label{thm:structure G/K'}
The quotient $\G/K'$ is isomorphic to the semidirect product
$\langle x \rangle \ltimes \Z^{p^n-1}$, where the element $x$ is of order $p^n$ and acts on $\Z^{p^n-1}$ via the companion matrix of the polynomial $X^{p^n-1}+X^{p^n-2}+\cdots+X+1$.
\end{theorem}

\begin{proof}
Let $P$ be the semidirect product in the statement of the theorem.
We first study the lower central series of $P$.
Set $V=\Z^{p^n-1}$ and write $(v_1,\ldots,v_{p^n-1})$ for the canonical basis of $V$.
Since we use right actions of groups, if $M$
is the companion matrix of $f(X)=X^{p^n-1}+X^{p^n-2}+\cdots+X+1$, then $v^x=vM$ for every $v\in V$.

For every $W\le V$ that is normal in $P$, we have $[W,P]=\{[w,x] \mid w\in W\}$, since $V$ is abelian and the map $v\mapsto [v,x]$ is a homomorphism on $V$.
Since $P/V$ is cyclic, we have $P'=[V,P]$, and consequently
\[
\gamma_i(P) = \{ [v,x,\overset{i-1}{\ldots},x] \mid v\in V \} = \{ v(M-I)^{i-1} \mid v\in V \},
\]
where $I$ stands for the identity matrix of order $p^n-1$.
Hence the rows of $(M-I)^{i-1}$ are generators of $\gamma_i(P)$.
Since $V$ is a free abelian group of finite rank, it follows that
\begin{equation}
\label{eqn:P/gamma_i(P)}
|P:\gamma_i(P)| = p^n \, |V:\gamma_i(P)| = p^n \, \det (M-I)^{i-1} = p^n f(1)^{i-1} = p^{in}
\end{equation}
for every $i\ge 2$, since $f(X)$ is the characteristic polynomial of $M$.
Observe also that $\cap_{i\ge 1} \, \gamma_i(P)=1$, since $P$ is residually a finite $p$-group.

Now recall that $K=\langle y_j \mid j=0,\ldots,p^n-1 \rangle$ with $y_j^a=y_{j+1}$ for all $j$.
In particular $y_{p^n-2}^a=y_{p^n-1}=y_0^{-1}\ldots y_{p^n-2}^{-1}$.
Hence the assignments $x\mapsto a$ and $v_i\mapsto y_{i-1}$ define a homomorphism $\alpha$ from $P$ onto $Q=\G/K'$.
Suppose that $1\ne w\in\ker\alpha$ and let $\ell\ge 1$ be such that
$w\not\in \gamma_{\ell+1}(P)$.
Then $\alpha$ induces an epimorphism from $P/\gamma_{\ell+1}(P)$ onto
$Q/\gamma_{\ell+1}(Q)$ whose kernel is not trivial, and consequently
\[
|P:\gamma_{\ell+1}(P)| > |Q:\gamma_{\ell+1}(Q)| \ge |Q_{\ell}:\gamma_{\ell+1}(Q_{\ell})|,
\]
where $Q_{\ell}=\G/K'\st_{\G}(\ell)$.
This is a contradiction, since $|P:\gamma_{\ell+1}(P)|=p^{(\ell+1)n}$ by
(\ref{eqn:P/gamma_i(P)}) and
$|Q_{\ell}:\gamma_{\ell+1}(Q_{\ell})|=|Q_{\ell}|=p^{(\ell+1)n}$ by Lemma~\ref{lem:G/Kstab l}.
Thus $\ker\alpha=1$ and we conclude that $P\cong \G/K'$, as desired.
\end{proof}

Now we can generalise \cite[Theorem~3.7]{Fernandez-Alcober2017} and show that $\G$ is not a branch group.

\begin{theorem}
Let $\G$ be a GGS-group  with constant defining vector.
Then $\G$ is not a branch group.
\end{theorem}

\begin{proof}
Let $L=\psi^{-1}(K'\times\overset{p^n}{\cdots} \times K')$.
By (ii) of Lemma~\ref{lem: gen properties of GGS constant} we have $L\le \rst_{\G'}(1)$.
For the reverse inclusion, we show that $\rst_{\G'}(v)\le L$ for every vertex $v$ of the first level.
Let $g\in \rst_{\G'}(v)$ and let $h$ be the component of $\psi(g)$ at the position of $v$.
Then from (i) of Proposition~\ref{pro:prod of g_i} we get $h\in K'$, and so $g\in L$, as desired.

Now assume by way of contradiction that the group $\G$ is branch.
Then $|\G:\rst_{\G}(1)|$ is finite and from \cite[Lemma~3.6]{Fernandez-Alcober2017} 
also $|\G':\rst_{\G'}(1)|$ is finite.
Thus $|\G:L|$ is finite, and since $L\le K'$ by Lemma~\ref{lem: gen properties of GGS constant}, also
$|\G:K'|$ is finite.
This is a contradiction, since Theorem~\ref{thm:structure G/K'} shows that the factor group
$\G/K'$ is infinite.
\end{proof}

\bibliographystyle{amsplain}

\bibliography{biblio1}

\end{document}